\documentclass[12pt]{article}

\usepackage{amsmath,amssymb,amsthm}
\usepackage[applemac]{inputenc}
\usepackage{amsmath,amssymb,fullpage}
\usepackage{color}

\newtheorem{definition}{Definition}[section]

\newtheorem{theorem}[definition]{Theorem}

\newtheorem{remark}[definition]{Remark}

\numberwithin{equation}{section}

\def\cF{{\mathcal {F}}}
\def\RR{{\mathbb{ R}}} 
\def\cB{{\mathcal{B}}}
\def\EE{{\mathbb{ E}}}

\def\1B{\text{1\!\!I}}

\begin{document}
\date{1 August 2017}
\title{Linear Volterra backward stochastic differential equations}

\author{
Yaozhong Hu$^{1,2}$ and Bernt \O ksendal$^{2}$   }

\footnotetext[1]{ {Department of Mathematical 
and Statistical Sciences,     University of Alberta,  Edmonton,  Alberta, 
Canada T6G 2G1.
Email: {\tt yaozhong@ualberta.ca}. \\
Y. Hu is partially supported by a grant from the Simons Foundation \#209206 and \#523775.}}
\footnotetext[2]{Department of Mathematics, University of Oslo, P.O. Box 1053 Blindern, N--0316 Oslo, Norway.\\
  Email: {\tt oksendal@math.uio.no}.\\
This research was carried out with support of the Norwegian Research Council, within the research project Challenges in Stochastic Control, Information and Applications (STOCONINF), project number 250768/F20.
 }

\maketitle



\begin{abstract} We present an  explicit solution triplet 
$(Y, Z, K)$  to the  backward stochastic Volterra integral equation (BSVIE) of linear type, driven by a Brownian motion and a compensated Poisson random measure.  The process $Y$ is expressed by   an  integral 
whose kernel is explicitly given. The processes $Z$ and $K$ are expressed by
Hida-Malliavin derivatives involving $Y$.  
\end{abstract}


 \section{Introduction and main theorem}\label{sec1}
 Backward stochastic Volterra integral equations (BSVIEs) were introduced 
  to solve stochastic optimal control problems for controlled Volterra type systems. Yong (see e.g. \cite{Y1}, \cite{Y2})  gives a systematic study, including  the existence and uniqueness of the solution of general nonlinear equations.  
In general the classical backward stochastic differential equations are hard to solve, and those of Volterra type are even worse. In this paper we aim to present an explicit solution formula for linear BSVIEs.  

To present our main result we start with     a basic  probability space
$(\Omega, \mathcal{F}, \mathbb{P})$  
equipped with a filtration 
$\mathbb{F}=\left\{\mathcal{F}_t \right\}_{0\le t\le T}$ 
satisfying the usual condition.  
Let $\left(B(t), 0\le t\le T\right)$ be a  one 
dimensional  Brownian motion  and 
let $\left\{ N(t, A)\,, 0\le t\le T\,, A\in \cB(\RR)\right\}$ be an independent Poisson random measure with Lévy measure $\nu$.  Both the Brownian motion $B$ 
and Poisson random measure $N$ are adapted to  the filtration $\mathbb{F}$.  
Let $\left\{F(t)\,, 0\le t\le T \right\}$ be a given stochastic 
process, not necessarily adapted. 

Let $\left(\Phi(t,s) , 0\le t<s\le T\right)$  and $(\xi(s), \beta(s,\zeta); 0\le s\le T\,, \zeta\in \RR)$  be given measurable functions 
of $t$, $s$, and $\zeta$,  with values in $\mathbb{R}$. For simplicity we assume that these functions are bounded, and we assume that there exists $\varepsilon > 0$ such that $\beta(s,\zeta) \geq -1 + \varepsilon$ for all $s,\zeta$.
We consider the following  linear backward stochastic Volterra integral equations in the unknown process triplet $(Y(t), Z(t,s),K(t,s,\zeta))$:
 \begin{align}
 Y(t)&= F(t)+\int_t^T \left[ \Phi(t,s) Y(s)+\xi(s) Z(t,s)+\int_{\mathbb{R}}\beta(s,\zeta)K(t,s,\zeta)\nu(d\zeta)\right] ds\nonumber\\
& -\int_t^T Z(t,s) dB(s)-\int_t^T\int_{\mathbb{R}^d}K(t,s,\zeta)\tilde{N}(ds,d\zeta)\,,
 \label{e.1.1}
 \end{align}
where $0\le t\le T$ and $\tilde N(dt,d\zeta)=N(dt,d\zeta)-\nu(d\zeta)dt$ is the compensated Poisson random measure.
  We want to find three  processes
$\left(Y(t),Z(t,s),K(t,s,\zeta)\,, 0\le t\le s\le T\,, \zeta\in \RR\right)$ which are adapted (meaning that 
$Y(s),Z(t,s),K(t,s,\zeta)\in \mathcal{F}_s$ for any $s\in [t, T]$ and any $\zeta \in \mathbb{R}$) such that the above equation \eqref{e.1.1} is satisfied.

To find such a solution, we   first  try to get rid of the unknowns  $Z(t,s)$ and $K(t,s,\zeta)$  inside      the first integral 
in \eqref{e.1.1}. 
To this end, 
we   define the probability measure $\mathbb{Q}$ by
\begin{equation}
d\mathbb{Q}=M(T)d\mathbb{P} \text{ on } \mathcal{F}_T,
\end{equation}
where
\begin{align}
M(t)&:= \exp \Big(\int_0^t \xi(s)dB(s) - \frac{1}{2}\int_0^t \xi^2(s)ds +\int_0^t\int_{\mathbb{R}} \ln (1+\beta(s,\zeta))\tilde{N}(ds,d\zeta) \nonumber\\
&+\int_0^t \int_{\mathbb{R}} \{ \ln (1+\beta(s,\zeta)) - \beta(s,\zeta) \} \nu(d\zeta)ds \Big); \quad 0 \leq t \leq T.
\end{align}
Then (see e.g. \cite{OS} Ch.1)
under the  probability measure $\mathbb{Q}$ the process
\begin{equation}
B_{\mathbb{Q}}(t):=B(t)-\int_0^t \xi(s) ds\,,\quad 0\le t\le T\,. 
\label{e.def_bq} 
\end{equation}
 is a Brownian motion, and the random measure
 \begin{equation}
 \tilde{N}_{\mathbb{Q}}(dt,d\zeta):=\tilde{N}(dt,d\zeta)-\beta(t,\zeta)\nu(d\zeta)dt
 \label{e.def_nq} 
 \end{equation}
 is the $\mathbb{Q}$-compensated Poisson random measure of $N(\cdot,\cdot)$, in the sense that the process
 
$$\tilde{N}_{\gamma}(t):=\int_0^t \int_{\mathbb{R}} \gamma(s,\zeta) \tilde{N}_{\mathbb{Q}}(ds,d\zeta)$$
 is a local $\mathbb{Q}$-martingale, for all predictable processes $\gamma(t,\zeta)$ such that
 \begin{equation}
 \int_0^T \int_{\mathbb{R}} \gamma^2(t,\zeta)\beta^2(t,\zeta) \nu(d\zeta)dt < \infty.
 \end{equation}

 We also introduce, for $0 \leq t \leq   r  \leq T$,
\[
\Phi^{(1)}(t,r)=\Phi(t,r)\,,\quad 
\Phi^{(2)}(t,r)=\int_t^r \Phi(t,s) \Phi(s,r) ds 
\]
and  inductively
\begin{equation}
\Phi^{(n)}(t,r)=\int_t^r \Phi^{(n-1)} (t,s) \Phi(s,r) ds \,,\quad  n=3, 4, \cdots\,. 
\end{equation}
\begin{remark} Note that if $|\Phi(t,r)| \leq C$ (constant) for all $t,r$, then by induction
\begin{align}
|\Phi^{(n)}(t,r)| \leq \frac{C^n T^n}{n !}
\end{align}
for all $t,r,n$. Hence,
\begin{align}
  \sum_{n=1}^{\infty} | \Phi^{(n)}(t,r)| < \infty
\end{align}
for all $t,r$.
\end{remark}
With these notations, we can state our main theorem of this paper as follows:

\begin{theorem} 
Put 
\begin{equation}
  \Psi(t,r)= \displaystyle \sum_{n=1}^\infty  \Phi^{(n)}(t,r)\,. \label{e.def_psi}
  \end{equation}Then we have the following explicit form of the solution triplet:
\begin{itemize}
\item[(i)]\ The $Y$ component of the solution triplet is given by 
\begin{align} \label{eq1.10}
Y(t)  &=\mathbb{E}_{\mathbb{Q}}\left[ F(t)\Big|\mathcal{F}_t\right]+  \int_t^T  \Psi(t,r) \mathbb{E}_{\mathbb{Q}}\left[ F(r) \Big|\mathcal{F}_t\right] dr\nonumber \\
&=\mathbb{E}_{\mathbb{Q}}\left[ F(t) +  \int_t^T  \Psi(t,r)   F(r)dr  \Big|\mathcal{F}_t\right] \,.
\end{align}  
\item[(i)] The $Z$ and $K$ components of the solution triplet are given by the following:\\
Define 
\begin{equation}\label{eq1.11}
U(t)= F(t)+\int_t^T  \Phi(t,r) Y(r)  dr-Y(t);\quad 0 \leq t \leq T.
\end{equation}
Then $Z(t,s)$ and $K(t,s,\zeta)$ can be expressed by the Hida-Malliavin derivatives $D_s$ and $D_{s,\zeta}$ with respect to $B$ and $N$, respectively, as follows:
\begin{equation} \label{eq1.14}
Z(t,s) = \mathbb{E}_{\mathbb{Q}}[D_sU(t)-U(t)\int_s^T D_s\xi(r)dB_{\mathbb{Q}}(r)|\mathcal{F}_s];\quad     
0\le   t \leq s \leq T 
\end{equation}
and
\begin{equation}\label{eq1.15}
K(t,s,\zeta) = \mathbb{E}_{\mathbb{Q}}[U(t)(\tilde{H}_s-1)+\tilde{H}_sD_{s,\zeta}U(t) | \mathcal{F}_s];\quad 
    0\le  t \leq s \leq T,
\end{equation}
where 
\begin{align}
\tilde{H}_s&= \exp \Big[ \int_0^s \int_{\mathbb{R}}[D_{s,x}\beta(r,x) + \log (1-\frac{D_{s,x} \beta(r,x)}{1-\beta(r,x)})(1-\beta(r,x) ]\nu(dx)dr \nonumber\\
&+ \int_0^s \int_{\mathbb{R}}\log (1-\frac{D_{s,x} \beta(r,x)}{1-\beta(r,x)} \tilde{N}_{\mathbb{Q}}(dr,dx) \Big].
\end{align}
\end{itemize}
\end{theorem} 
 
\begin{proof}
With the processes  $B_{\mathbb{Q}}$ and $\tilde{N}_{\mathbb{Q}}$ 
defined in \eqref{e.def_bq}-\eqref{e.def_nq} we can
eliminate the unknowns   $Z(t,s)$ and $K(t,s,\zeta)$  inside      the first integral 
in \eqref{e.1.1}. More precisely,  
we can  rewrite  equation \eqref{e.1.1} as
\begin{equation}
 Y(t)= F(t)+\int_t^T  \Phi(t,s) Y(s)  ds-\int_t^T Z(t,s) dB_{\mathbb{Q}}(s)-\int_t^T \int_{\mathbb{R}}K(t,s,\zeta)\tilde{N}_{\mathbb{Q}}(ds,d\zeta)\,, 
 \label{e.1.2}
 \end{equation}
where $ 0\le t\le T$.  Taking the conditional $\mathbb{Q}$-expectation on $\mathcal{F}_t$, we get
\begin{eqnarray}
Y(t)&=& \mathbb{E}_{\mathbb{Q}}\left[ F(t)+\int_t^T  \Phi(t,s) Y(s)  ds\big|\mathcal{F}_t\right]\nonumber\\
 &=& \tilde F(t,t)+\int_t^T  \Phi(t,s) \mathbb{E}_{\mathbb{Q}}\left[Y(s)   \big|\mathcal{F}_t\right]ds
\,,\quad 0\le t\le T\,.\label{1.1.5}
\end{eqnarray}
Here, and in what follows, we denote
\begin{equation}
\tilde{F}(t,s)=\mathbb{E}_{\mathbb{Q}}\left[F(t) \big|\mathcal{F}_s\right]\,. 
\end{equation}
Fix $r\in [0,  t]$.   Taking the conditional $\mathbb{Q}$-expectation on $\mathcal{F}_r$ of \eqref{1.1.5}, we get
\begin{eqnarray}
\EE_Q \left[Y(t)|\cF_r\right]= \tilde F(t,r)+\int_t^T  \Phi(t,s) \mathbb{E}_Q\left[Y(s)   \big|\mathcal{F}_r\right]ds
\,,\quad r\le t\le T\, 
\end{eqnarray}
Denote 
\[
\tilde Y(s)=\EE_Q \left[Y(s)|\cF_r\right]\,,\quad r\le s\le T\,.
\]
Then the above equation can be written as 
\[
\tilde Y(t)=\tilde F(t,r)+\int_t^T  \Phi(t,s) \tilde Y(s) ds
\,,\quad r\le t\le T\,.
\]
Substituting $ 
\tilde Y(s)=\tilde F(s,r)+\int_s^T  \Phi(s,u) \tilde Y(u) du$ in the above equation,  we obtain
\begin{eqnarray}
\tilde Y(t)
 &=& \tilde{F}(t,r)+\int_t^T  \Phi(t,s)\left\{\tilde{F}(s,r) 
 +\int_s^T  \Phi(s,u)  \tilde  Y(u)   du\right\} ds\nonumber\\
 &=&  \tilde{F}(t,r)+\int_t^T  \Phi(t,s) \tilde{F}(s,r) ds 
 +\int_t^T  \Phi^{(2)} (t,u)  \tilde Y(u) du  ds
\,,\quad r\le t\le T\,, \nonumber
\end{eqnarray}   

By  repeatedly using the above argument, we get
\begin{eqnarray}
\tilde Y(t) 
 &=&  \tilde F(t,r)+\sum_{n=1}^\infty \int_t^T  \Phi^{(n)}(t,u) \tilde F(u,r) du  \nonumber\\
 &=& \tilde F(t,r)+  \int_t^T  \Psi(t,u) \tilde F(u,r) du \,,
\end{eqnarray}   
where  $\Psi$ is defined by \eqref{e.def_psi}. 
Now substituting $\EE_Q(Y(s)|\cF_t)=\tilde Y(s) $ (with $r=t$)
into \eqref{1.1.5} we obtain part (i) of the theorem. 
 
It remains to prove \eqref{eq1.14}-\eqref{eq1.15}. 
By \eqref{e.1.2} we have
\begin{equation}
U(t)= \int_t^T Z(t,s)dB_Q(s) + \int_t^T \int_{\mathbb{R}}K(t,s,\zeta)\tilde{N}_Q(ds,d\zeta); \quad 0\leq t\leq s\leq T.
\end{equation}
Note that by the Clark-Ocone formula  under change of measure (see \cite {H},  \cite{O1}), extended to $L^2(\mathbb{Q},\mathcal{F}_T)$ as in \cite{AaOPU},
we get 
\begin{equation}
Z(t,s) = \mathbb{E}_Q[D_sU(t)-U(t)\int_s^T D_s\xi(r)dB_Q(r)|\mathcal{F}_s];\quad t \leq s \leq T 
\end{equation}
and
\begin{equation}
K(t,s,\zeta) = \mathbb{E}_Q[U(t)(\tilde{H}_s-1)+\tilde{H}_sD_{s,\zeta}U(t) | \mathcal{F}_s];\quad t \leq s \leq T 
\end{equation}
where 
\begin{align}
\tilde{H}_s&= \exp \Big[ \int_0^s \int_{\mathbb{R}}[D_{s,x}\beta(r,x) + \log (1-\frac{D_{s,x} \beta(r,x)}{1-\beta(r,x)})(1-\beta(r,x) ]\nu(dx)dr \nonumber\\
&+ \int_0^s \int_{\mathbb{R}}\log (1-\frac{D_{s,x} \beta(r,x)}{1-\beta(r,x)} \tilde{N}_Q(dr,dx) \Big],
\end{align}
as claimed.
\end{proof} 
 
\section{Application to smoothness of the solution triplet}
It is of interest to study when the solution components $Z(t,s), K(t,s,\zeta)$ are smooth ($C^1$) with respect to $t$.  Such smoothness  properties  are important in the  study of
optimal control (see e.g. \cite{AOY}). It is also important
in the numerical solutions (see e.g. \cite{hns} and references therein).     Using the explicit form of the solution triplet  (Theorem 1.2) we can give sufficient conditions for such smoothness in the linear case.

\begin{theorem}
Assume that $\xi,\beta$ are deterministic and that $F(t)$ and $\Phi(t,s)$ are $C^1$ with respect to $t$ satisfying
\begin{equation}
\mathbb{E}_{\mathbb{Q}}\Big[\int_0^T (\int_t^T \Big\{ F^2(t)+\Phi^2(t,s)+(\frac{dF(t)}{dt})^2 + (\frac{\partial \Phi(t,s)}{\partial t})^2 \Big\}ds)dt \Big] < \infty.
\end{equation}
Then, for $t < s ≤ T,$ 
\begin{align}
Z(t,s)&=\mathbb{E}_{\mathbb{Q}}[D_sF(t) + \int_t^T D_s(\phi(t,r)Y(r)dr) | \mathcal{F}_s],\\
K(t,s,\zeta)&=\mathbb{E}_{\mathbb{Q}}[D_{s,\zeta}F(t) + \int_t^T D_{s,\zeta}(\phi(t,r)Y(r)dr) | \mathcal{F}_s].
\end{align}
In particular, we have 
\begin{equation}
\mathbb{E}_{\mathbb{Q}}\left[  \int_{0}^{T} (\int_{t}^{T}\left(  \frac{\partial Z}{\partial
t}\left(  t,s\right)  \right)  ^{2}ds) dt+\int_{0}^{T} (\int_{t}^{T}\int_{
\mathbb{R}
}\left(  \frac{\partial K}{\partial t}\left(  t,s,\zeta\right)  \right)  ^{2}
\nu\left(  d\zeta\right)  ds)dt\right]  <\infty.
\end{equation}

\end{theorem}

\begin{proof}
Since $Y(t)$ is $\mathcal{F}_t$-measurable, we get that $D_sY(t)=D_{s,\zeta}Y(t)=0$ for all $s>t$. Hence by \eqref{eq1.11}
\begin{equation}
\mathbb{E}_{\mathbb{Q}}[D_sU(t)|\mathcal{F}_s]= \mathbb{E}_{\mathbb{Q}}[D_sF(t) +\int_t^T D_s(\Phi(t,r)Y(r))dr | \mathcal{F}_s]
\end{equation}
and
\begin{equation}
\mathbb{E}_{\mathbb{Q}}[D_{s,\zeta}U(t) | \mathcal{F}_s]= \mathbb{E}_{\mathbb{Q}}[D_{s,\zeta}F(t) +\int_t^T D_{s,\zeta}(\Phi(t,r)Y(r))dr | \mathcal{F}_s].
\end{equation}
Then the result follows from \eqref{eq1.14} and \eqref{eq1.15}.
 \end{proof}


\begin{thebibliography}{99}   

\bibitem {AaOPU}Aase, K., Øksendal, B., Privault, N. and Ubøe, J.  White noise
generalizations of the Clark-Haussmann-Ocone theorem with application to
mathematical finance. Finance Stochast. 4 (2000), 465-496.                                                                                            

\bibitem {AO}Agram, N. and \O ksendal, B.  Malliavin calculus and optimal control
of stochastic Volterra equations. J. Optim. Theory Appl. \     
DOI10.1007/s10957-015-0735-5 (2015).

\bibitem{AOY} Agram, N., \O ksendal, B. and Yakhlef, S. 
Optimal control of forward-backward stochastic Volterra equations.
In \emph{F. Gesztezy et al (editors): Partial Differential equations, Mathematical Physics, and Stochastic Analysis. A Volume in Honor of Helge Holden's 60th Birthday}. EMS Congress Reports.
http://arxiv.org/abs/1669.409v2

\bibitem{hns} Hu, Y.;  Nualart, D.  and  Song, X. 
 Malliavin calculus for  backward
stochastic differential equations and application to
numerical schemes.  The Annals of Applied Probability
 Vol. 21 (2011),  no. 6,  2379-2423.
 
\bibitem {H}Huehne, F.  A Clark-Ocone-Haussmann formula for optimal portfolio under Girsanov transformed pure-jump L\' evy processes. Working paper 2005.


\bibitem {O1}Okur, Y. Y.  An extension of the Clark-Ocone formula under change of measure for L\' evy processes. Preprint Series, Dept. of Mathematics, University of Oslo January 2008.

\bibitem{OS}
{\O}ksendal, B. and Sulem, A.  \newblock
\emph{Applied Stochastic Control of Jump Diffusions}. \newblock
Second Edition, Springer 2007.


%
%
%
%
%
%
%
%

\bibitem {Y1}Yong, J.   Backward stochastic Volterra equations and some related
problems, Stochastic Processes and their Applications 116 (2006), 779-795.

\bibitem {Y2}Yong, J.   Backward stochastic Volterra integral equations - a
brief survey. Appl. Math. J. Chinese Univ. 28(4) (2013), 383-394.

\end{thebibliography}
\end{document}